\documentclass[10pt]{article}

\usepackage{arxivlatex}

\title{\bfseries Compressive Sensing with a Multiple Convex Sets Domain}
\author{Hang Zhang, Afshin Abdi, and Faramarz Fekri \\
{\small School of Electrical and Computer Engineering, Georgia Institute of Technology, Atlanta, GA.}}

\date{}

\begin{document}
\maketitle 

%===============================

\begin{abstract}
In this paper, we study a general framework 
for compressive sensing assuming the existence of the prior knowledge 
that $\bx^{\natural}$ belongs to the union of multiple convex sets, 
$\bx^{\natural} \in \bigcup_i \calC_i$. 
In fact, by proper choices of these convex sets in the above 
framework, the problem can be transformed to well known CS 
problems such as the phase retrieval, quantized compressive 
sensing, and model-based CS.
First we analyze the impact of this prior knowledge
on the minimum number of measurements $M$ to guarantee the uniqueness of the solution.
Then we formulate
a universal objective function for signal recovery, which 
is both computationally inexpensive and flexible. 
Then, an algorithm based on \emph{multiplicative weight update}
and \emph{proximal gradient descent} is proposed 
and analyzed for signal reconstruction. 
Finally, 
we investigate as to how we can improve the signal recovery 
by introducing regularizers into the objective function. 
\end{abstract}

\section{Introduction}
%==============================
In the traditional \emph{compressive sensing} (CS) 
\cite{foucart2013mathematical, eldar2012compressed}, 
sparse signal $\bx$ is reconstructed via 
\begin{equation}
\label{eqn:trad_cs}
\min_{\bx}~\|\bx\|_1,~~\St~~\by = \bA\bx, 
\end{equation}
where $\by \in \RR^M$ denotes the measurement vector and 
$\bA \in \RR^{M\times N}$ is the measurement matrix. 
\par 
In this paper, we assume the existence of extra prior knowledge 
that $\bx$ lies in the union of some convex sets, 
$\bx\in \bigcup_{i=1}^L \mathcal{C}_i$, 
where $L$ denotes the number of constraint sets and $\calC_i$ is the $i$-th 
convex constraint set. 
Therefore, we now wish to solve
\begin{equation}
\begin{aligned}
\min_{\bx}\|\bx\|_1,\
\St~~\by = \bA\bx, \quad \bx\in \bigcup_{i=1}^L \mathcal{C}_i
\end{aligned}
\label{eqn:multiconvex_cs}
\end{equation}
%we would like to incorporate the prior information, 
%namely, $\bx$ lies in the union of convex sets 
%$\bx\in \bigcup_{i=1}^L \mathcal{C}_i$, into the model, 
%where $L$ denotes the number of constraint sets and $C_i$ denotes the 
%corresponding convex constraint set. 
This ill-posed inverse problem (i.e., given measurement $\by$, 
solving for $\bx$) turns out to be a rather general form of 
CS. 
For example,
setting $\bigcup \mathcal{C}_i = \RR^n$ simplifies our 
problem to the traditional CS problem. 
In the following, we will further show that 
by appropriate choices of these convex sets, Eq.~(\ref{eqn:multiconvex_cs})
can be transformed to the \emph{phase retrieval} \cite{candes2011phase, chen2015solving}, 
\emph{quantized compressive sensing} \cite{dai2009quantized}, 
or \emph{model-based CS} \cite{baraniuk2010model} problems.

%==============================
\subsection{Relation with other problems} 
\paragraph{Phase retrieval}
Consider the noiseless phase retrieval problem in which the measurements are given by 
\[
y_i = \left|\langle \ba_i,~\bx\rangle\right|^2,~~1\leq i\leq l,  
\]
where $y_i$ is the $i$-th measurement and $\ba_i$ denotes 
the 
corresponding coefficients. Considering the first measurement, the constraint
$\sqrt{y_1} = |\langle  \ba_1,~\bx\rangle|$ can be represented 
via $\bx\in \calB_+^{(1)} \bigcup \calB_-^{(1)}$ where 
$\calB_+^{(1)} = \{\bx:~\langle \ba_1, \bx\rangle = \sqrt{y_1}\}$ and 
$\calB_-^{(1)} = \{\bx:~\langle \ba_1, \bx \rangle = -\sqrt{y_1}\}$. 
Following these steps, 
the constraints
$\{y_i = \langle \ba_i, \bx\rangle ^2\}_{i=1}^l$ 
can be transformed to
$\bx\in \bigcap_i \big(\calB_+^{(i)} \bigcup \calB_-^{(i)}\big)
= \bigcup_{j=1}^{2^l}\mathcal{C}_j$, 
for some appropriately defined $C_j$'s given by the 
intersection of different $\+B_{\pm}^{(i)}$.
Setting sensing matrix 
$\bA = \bZero$ will restore the phase retrieval to our setting. 
%==============================
\paragraph{Quantized compressive sensing}
In this scenario, the measurements are quantized, i.e., 
%sensor readout is a quantized version of the true readout, which means 
\[
y_i = Q(\langle \ba_i,~\bx\rangle),~1\leq i \leq L,  
\]
where $Q(\cdot)$ is the quantizer. 
Since $Q^{-1}(\cdot)$ is an interval on real line, $\mathcal{C}_i$ would be a 
convex set and the quantized CS can be easily transformed to 
Eq.~(\ref{eqn:multiconvex_cs}).
%==============================
\paragraph{Model-based compressive sensing}
These lines of works \cite{baraniuk2010model, duarte2011structured, silva2011blind}
are the most similar work to our model, where they consider
 
\[
\by = \bA\bx^{\natural},~~\bx^{\natural}\in \bigcup_{i}\calL_i.
\]
Here, $\calL_i$ is assumed to be a linear space whereas 
the only assumption we make on the models is being a convex set.
Hence, their model can be regarded as a special case of our problem.
\par  
In \cite{baraniuk2010model}, the author studied the minimum number of measurements
$M$ under different models, i.e., shape of $\calL_i$, and modified 
CoSaMP algorithms \cite{foucart2013mathematical} to reconstruct signal. 
In \cite{duarte2011structured}, the authors expanded the signal onto different basis
and transformed model-based CS to be 
block-sparse CS. In \cite{silva2011blind}, the author 
studied model-based CS with incomplete sensing matrix information and 
reformulated it as a matrix completion problem.

%This line of work\cite{baraniuk2010model, foucart2013mathematical} is the 
%most similar work of our model. 
%While our model 
%only requires $\calC_i$ to be convex, 
%\cite{baraniuk2010model, foucart2013mathematical} 
%also needs
%them to be subspaces, which makes their model a special case of 
%ours. 

%==============================
\subsection{Our contribution:}
%Our contribution can be divided into two parts in general:
%the statistical property and associated optimization method. 

\par 
\paragraph{Statistical Analysis}
We analyze the minimum number of measurements 
to ensure uniqueness of the solution. 
We first show that the conditions for the uniqueness
can be represented as 
$\min_{u\in E}\|\bA\bu\|_2 > 0$, for an appropriate set $E$. 
Assuming the entries of the sensing matrix $\bA$ are i.i.d. Gaussian, 
we relate the probability of uniqueness to the number of measurements, $M$. 
Our results show that depending on the structure of $\calC_i$'s, 
the number of measurements can be reduced significantly.

%=========================================
\paragraph{Optimization Algorithm}
We propose a novel formulation and the associated 
optimization algorithm to reconstruct 
the signal $\bx$. 
First, note that existing algorithms on e.g., model-based CS
are not applicable to our problem as they rely heavily on the structure of 
constraint sets. 
For example, a key idea in model-based CS is 
to consider expansion of $\bx$ onto the basis of each $\mathcal{C}_i$ and 
then rephrase the constraint as the block sparsity on the representation of $\bx$ 
on the union of bases. 
However, such an approach may add complicated constraints on the coefficients of $\bx$ 
in the new basis, 
as the sets $\calC_i$'s are not necessarily simple subspaces.
\par 
Note that although $\calC_i$'s are assumed to be convex, 
their union $\bigcup_{i}\calC_i$ is not necessarily a convex set, 
which makes the optimization problem Eq.~(\ref{eqn:multiconvex_cs}) hard to solve. 
By introducing an auxiliary variable, $\bp$, we convert the non-convex 
optimization problem to a biconvex problem. 
Using \emph{multiplicative weight update} \cite{arora2012multiplicative}
from online learning theory \cite{shalev2012online}, we 
design an algorithm with convergence speed of 
$\calO(T^{-1/2})$ to a local minimum. Further, 
we investigate improving the performance of the algorithm by
incorporating appropriate regularization.
Compared to the naive idea of solving $L$ simultaneous optimization problems

\[
\min_{\bx}~\|\bx\|_1,~\St~\by = \bA\bx,\quad \bx\in \calC_i,
\]
and choosing the best solution out of $L$ results,  
our method is computationally less-expensive and more flexible.
%==============================

\section{System Model}
%==============================
Let $\bA\in \RR^{M\times N}$ be the measurement matrix, and
consider the setup  
\begin{equation}
\label{eqn:cs_multi_set_model}
\by = \bA\bx^{\natural},\ \ \ \mbox{and} ~~~\bx^{\natural}\in \bigcup_{i=1}^L \calC_i,
\end{equation}
where $\bx^{\natural}$ is a $K$-sparse high-dimensional signal, 
$\by\in \RR^M$ is the measurement vector,  
and $\calC_i\subset\RR^N$, $i=1,2,\ldots,L$, is a convex set.
\par 
Due to the sparsity of $\bx^{\natural}$, 
we propose to reconstruct $\bx^{\natural}$ via 
\begin{equation}
\label{eqn:cs_mutli_set_optim}
\wh{\bx} = \argmin_{\bx}~\|\bx\|_1,\
~~\textup{s.t.}~~\by = \bA\bx,\
~~ \bx\in \bigcup_{i=1}^L \mathcal{C}_i,
\end{equation}
where $\wh{\bx}$ denotes the reconstructed signal.  
Let $\bd \defequal \wh{\bx} - \bx^{\natural}$ be the 
deviation of the reconstructed signal $\wh{\bx}$
from the true signal $\bx^{\natural}$.
In the following, we will study the inverse problem 
in Eq.~(\ref{eqn:cs_mutli_set_optim}) from two perspectives; 
the statistical and the computational aspects.

\section{Statistical Property}
%===================================
In this section, we will find the minimum number of measurements 
$M$ to  $\wh{\bx} = \bx^{\natural}$, i.e., $\bd = \bZero$.

\begin{definition}
The tangent cone $\mathcal{T}_{\bx}$ for $\|\bx\|_1$ 
is defined as \cite{chandrasekaran2012convex}
\[
\mathcal{T}_{\bx} \triangleq \{\be:~\
\|\bx + t\be\|_1 \leq \|\bx\|_1,~\exists~t \geq 0\}.
\]	
\end{definition}
Geometric interpretation of $\mathcal{T}_{\bx}$ 
is that it contains all directions that lead to smaller $\|\cdot\|_1$ originating 
from $\bx$. In the following analysis,
we use $\mathcal{T}$ as a compact notation for $\mathcal{T}_{\bx^{\natural}}$.
Easily we can prove that $\bd\in \calT$.

\begin{definition}
The Gaussian width $\omega(\cdot)$ associated with 
set $U$ is defined as
$
\omega(U) \defequal \Expc \sup_{\bx \in U} \la \bg,~\bx\ra,
\bg \sim \normdist(\bZero, \bI)$,  \cite{gordon1988milman}.
\end{definition}

\par
Define cone $\wt{\calC}_{i,j}$ as 
\[
\wt{\calC}_{i,j} \defequal \ 
\left\{\bz~\big|~\bz = t(\bx_1 - \bx_2),~\exists~t > 0,~\bx_1 \in \calC_i, \bx_2 \in \
\calC_{j}\right\},
\]
which denotes the cone consisting of all vectors $\bz$ that are parallel 
with $\bx_1 - \bx_2$, $\bx_1\in \calC_i$, and $\bx_2\in \calC_j$.
Then we define event $\calE$ as
 
\[
\calE \defequal \
\left\{\bigcup_{i,j}\left(\Null(\bA)\bigcap \calT \bigcap  \wt{\calC}_{i, j}\right) = \{\bZero\}\right\}.
\]

\begin{restatable}{lemma}{uniquesollemma}
\label{lemma:uniquesol}
We can guarantee the correct recovery of $\bx$, i.e., $\wh{\bx} = \bx^{\natural}$, iff we have event $\calE$ to be satisfied.
\end{restatable}

\begin{proof}
This proof is fundamentally same as 
\cite{hzhangEnergy, chandrasekaran2012convex}.
First we prove that $\calE$ leads to $\wh{\bx} \neq \bx^{\natural}$.
Provided $\bd\defequal\wh{\bx} -\bx^{\natural} \neq \bZero$, we then have 
a $\wh{\bx} \neq \bx^{\natural}$ such that
$\left\|\wh{\bx}\right\|_1 \leq \left\|\bx^{\natural}\right\|_1$.
Setting $\be\parallel \bd$,
we hav a non-zero 
$\be\in\Null(\bA)\bigcap \calT \bigcap (\bigcup_{i,j}\wt{\calC}_{i, j})$,
which violates $\calE$. 
\par 
Then we prove that $\wh{\bx} \neq \bx^{\natural}$ implies 
$\calE$. 
Assume that there exists non-zero $\be\in\Null(\bA)\bigcap \calT \bigcap (\bigcup_{i,j}\wt{\calC}_{i, j})$. 
We can show that signal $\bx^{\natural} + t\be$, where $t$ is some positive 
constant such that $\|\bx^{\natural} + t\be\|_1 \leq \|\bx^{\natural}\|_1$, satisfying constraints described 
by Eq.~(\ref{eqn:cs_multi_set_model}). This implies 
that $\bd = t\be \neq \bZero$ and the wrong recovery of 
$\bx^{\natural}$.  
\end{proof}

Since a direct computation of the probability of event 
$\calE$ can be difficult, 
we analyze the following equivalent event, 

\[
\min_{\bx\in \calT \bigcap \left(\bigcup_{i,j}\wt{\calC}_{ij}\right)}
\|\bA\bx\|_2 > 0.
\]
For the simplicity of analysis, we assume that the entries 
$A_{i,j}$
of $\bA$ are i.i.d. normal $\normdist(0, 1)$.  
Using Gordon's escape from mesh theorem \cite{gordon1988milman}, 
we obtain the following result that relates $\Prob(\calE)$
with the number of measurements $M$.
\par \vsp 
%===================================
\begin{restatable}{theorem}{uniqthm}
\label{thm:cs_multi_set-uniq}
Let $a_M = \mathbb{E}\|\bg\|_2$, where $\bg\in \mathcal{N}(\boldsymbol{0},\bI_{M\times M})$, 
and $\omega(\cdot)$ denotes the Gaussian width. 
Provided that $a_M\geq \omega(\calT)$ and 
$(1-2\epsilon)a_M \geq \omega(\wt{\calC}_{ij})$ for
$1\leq i, j \leq L$ and $\epsilon > 0$, we have

\[
\begin{aligned}
\textup{Pr}(\calE)\
\geq 1 - \bigg(\underbrace{\textup{Pr}\bracket{\min_{\bu\in \mathcal{T}^c \setminus \{\boldsymbol{0}\} }\
\|\bA\bu\|_2 > 0}}_{\calP_1}\vcap \underbrace{\textup{Pr}\bracket{\
\min_{\boldsymbol{u}\in \bigcap \wt{\mathcal{C}}^c_i \setminus \{\boldsymbol{0}\}}\
\|\bA\bu\|_2 > 0}}_{\calP_2}\bigg),
\end{aligned}
\]
where $a \vcap b$ denotes the minimum of $a$ and $b$, 
and $\calP_1$ and $\calP_2$ can be bounded as 

\[
\begin{aligned}	
&\calP_1 \leq   1 \vcap \exp\bracket{-\frac{\left(a_M -\omega(\calT)\right)^2}{2}}\\
&\calP_2 \leq 1 \vcap \frac{3}{2}\exp\left(-\frac{\epsilon ^2a_M^2}{2}\right)
+ \sum_{i \leq j} \exp\bracket{-\dfrac{\bracket{(1-2\epsilon)a_M -\
\omega(\wt{\calC}_{ij})}^2}{2}}.
\end{aligned}
\]
\end{restatable}
%===================================

Thm.~\ref{thm:cs_multi_set-uniq} 
links the probability of correct recovery of 
Eq.~(\ref{eqn:cs_mutli_set_optim}) with the number of measurements $M$, 
and the ``size" of constraint set.
Detailed explanation is given as the following. 
To ensure high-probability of $\calE$, we would like 
to $\calP_1 \vcap \calP_2$ to approach zero, 
which requires large value of $a_M$. Meanwhile, 
$a_M$ is a monotonically increasing function of the sensor number $M$.
Hence, we can obtain the minimum sensor number $M$ 
requirement by unique recovery via investigating $a_M$. 

\begin{remark}
Notice that $\calP_1$ is associated with the 
descent cone $\calT$ of the optimization function, namely, 
$\|\bx\|_1$, while $\calP_2$ is associated with the 
prior knowledge $\bx\in \bigcup_i \calC_i$.
Thm.~\ref{thm:cs_multi_set-uniq} implies that
event $\calE$ (uniqueness) holds with higher 
probability than the traditional CS due to the 
extra constraint 
$\bx \in \bigcup_{i}\mathcal{C}_i$. 
If we fix $\textup{Pr}(\calE)$, 
we can separately calculate the corresponding $M$ with 
and without the constraint $\bx \in \bigcup_{i}\mathcal{C}_i$. The 
difference $\Delta M$ would indicate the savings in the number of measurements 
due to the 
additional structure $\bx\in \bigcup_i \mathcal{C}_i$ 
over the traditional CS. 
\end{remark}

One simple example is attached below 
to illustrate the improvement brought by 
Thm.~\ref{thm:cs_multi_set-uniq}.

\begin{example}
\label{example:support_set}
Consider the constraint set 
\[
\calC_i = \{(0, \cdots, 0, x_{i}, \cdots, x_{K + i},~0, \cdots, 0)\},
\]
where $1\leq i \leq N-K$.
%===============================
We study the 
asymptotic behavior of Thm.~\ref{thm:cs_multi_set-uniq}
when $N$ is of order $\calO(K^c)$, where $c > 1$ is constant.
In the sequel we will show that Thm.~\ref{thm:cs_multi_set-uniq} gives us the 
order $M = \calO(K)$ to ensure solution uniqueness 
as $K$ approaches infinity, which gives us the same bound 
as shown in \cite{baraniuk2010model} and suggests 
the tightness of our result.
\par 
Setting $\epsilon = 1/{4}$, we can bound $\calP_2$ as
\[ 
\calP_2 \leq \
\frac{3}{2}\exp\left(-\frac{a_M^2}{32}\right) + \
\dfrac{N^2}{2} \exp\left(-\frac{(a_M - 2a_{2K})^2}{8}\right),
\]
provided $a_M \geq 2a_K$. 
With the relation $\frac{M}{\sqrt{M+1}} \leq a_M \leq \sqrt{M}$ 
\cite{gordon1988milman, chandrasekaran2012convex} 
and setting $M = 3K$,
we have 
\[
\calP_2 \leq  c_1 \exp(-c_2 K) + c_3 N^2\exp(-c_4 K),
\]
where $c_1, c_2, c_3, c_4 > 0$ are some positive constants.
Since $N = \calO(K^c)$, we can see $\calP_2$ shrinks to zero 
as $K$ approaches infinity, which implies the solution uniqueness. 
\par 
Comparing with the traditional CS theory 
without prior knowledge $x \in \bigcup_i \calC_i$, 
our bound reduces the number of measurements 
from $M = \calO(K\log N/K) = \calO(K\log K)$ to 
$M = \calO(K)$.  
\end{example}

\section{Computational Algorithm}
%=================================
Apart from the statistical property, another important aspect of 
Eq.~(\ref{eqn:cs_mutli_set_optim}) is to design an efficient algorithm. 
One naive idea is to consider and solve $L$ separate optimization problems
\[
\wh{\bx}^{(i)} = \argmin_{\bx}~\|\bx\|_1,~\St~~\by = \bA\bx,
~~\bx\in \calC_i,
\] 
and then selecting the best one, i.e., the sparsest reconstructed signal
among all $\wh{\bx}^{(i)}$'s.
However, this method has two drawbacks: 

\begin{itemize}
\item 
It requires solving $L$ separate optimization problems, which in many 
applications might be prohibitively large and difficult to handle, but the 
proposed method is based on one single optimization procedure.

%====================================
\item 
It is inflexible. For example, some prior 
knowledge of which $\calC_i$ 
the true signal $\bx^{\natural}$ is more likely to reside might be available.
The above method cannot incorporate such priors.
\end{itemize}

\par 
To overcome the above drawbacks, we $(i)$ reformulate Eq.~(\ref{eqn:cs_mutli_set_optim})
to a more tractable objective function, 
and $(ii)$ propose a computationally efficient algorithm to solve it.
In the following, we assume that $\bx$ is bounded in the sense that 
for a constant $R$, $\|\bx\|_2 \leq R$.

\subsection{Reformulation of the objective function}
We introduce an auxiliary variable $\bp$ 
and rewrite the Lagrangian form in 
Eq.~(\ref{eqn:cs_mutli_set_optim}) as  

\begin{equation}
\label{eqn:cs_multi_set_model_unify}
\min_{\bx} \min_{\bp\in \Delta_L}\sum_{i}~p_i\
\bigg(\|\bx\|_1 + \wt{\Ind}(\bx\in \mathcal{C}_i)
+\dfrac{\lambda_1}{2}\|\by - \bA\bx\|^2_2 + \dfrac{\lambda_2\|\bx\|^2_2}{2}
\bigg),
\end{equation}
where $\Delta_L$ is the simplex $\{p_i \geq 0,~\sum_i p_i = 1\}$,
$\wt{\Ind}(\cdot)$ is the truncated indicator function, which is 
$0$ when its argument is true and %the condition is satisfied and 
is some large finite number $C$ otherwise, 
and $\lambda_1, \lambda_2 > 0$ are the Lagrange multipliers. 
The term $\|\by - \bA\bx\|^2_2$ is used to penalize for 
the constraint $\by = \bA\bx$ while 
$\|\bx\|^2_2$ corresponds to the energy constraint 
$\|\bx\|_2 \leq R$. 
It can be easily shown that solving Eq.~(\ref{eqn:cs_multi_set_model_unify}) 
for large enough $C$ ensures 
$\bx \in \bigcup_{i}^L \mathcal{C}_i$.

%====================================
\begin{algorithm}[H]
\caption{Non-convex Proximal Multiplicative Weighting Algorithm}
\label{alg:non_cvx_prox_mw}
\begin{algorithmic}[1]
\Statex \textbullet~
\textbf{Initialization}: Initialize all variables with uniform weight 
$p_i^{(0)} = L^{-1}$ and $\bx^{(0)} = \bZero$.
%==========================
\Statex \textbullet~\
\textbf{For time $t =1$ to $T$:}
We update $p_i^{(t+1)}$ and $\bx^{(t)}$ as 
\begin{align}
&p_{i}^{(t+1)} \propto p_i^{(t)} e^{-\eta_p^{(t)} f_i(\bx^{(t)})} \
\label{eqn:non_cvx_prox_p_update}\\
&\bx^{(t+1)} = \prox_{\eta^{(t)}_w \|\cdot\|_1}\
\bigg[\bx^{(t)} - \eta^{(t)}_{x} \sum_{i} p_i^{(t)} \notag \left(\nabla_x h_i(\bx^{(t)}) + \lambda_1 \bA^{\rmt}(\bA\bx^{(t)} - \by)\
+ \lambda_2 \bx^{(t)} \right)\bigg] \label{eqn:non_cvx_prox_x_update},
\end{align}
where $p_i^{(t)}$ denotes the $i$th element of 
$\bp^{(t)}$, 
and the proximal operator $\prox_{\|\cdot\|_1}(\bx)$ is defined as 
$\argmin_{\bz} \left(\|\bz\|_1 + \frac{1}{2}\|\bz-\bx\|^2_2\right)$
\cite{beck2009fast}.
%====================================

\Statex \textbullet ~\
\textbf{Output}:
Calculate the average value 
$\bar{\bp} = \frac{\sum_t \bp^{(t)}}{T}$ and 
value $\bar{\bx} = \frac{\sum_t \bx^{(t)}}{T}$.
Then output $\wh{\bx}$ by projecting 
$\bar{\bx}$ onto the set of $\bigcup_i \calC_i$. 

\end{algorithmic}
\end{algorithm}

%====================================
\par 
Apart from the universality, 
our formulation has the following benefits:
\begin{itemize}
\item 
It is memory efficient. Compared with the 
naive idea that needs to store $L$ different 
$\wh{\bx}^{(i)}$, our method only needs to 
track one $\wh{\bx}$ and one redundant variable $\bp$. 
This reduces the storing memory from 
$\calO(NL)$ to $\calO(N+L)$.
\item 
It is very flexible. We can easily 
adjust to the case that $\bx$ belongs to the
intersection, i.e., $\bx\in \bigcap_i \mathcal{C}_i$ via
modifying
$\min_{\bp\in \Delta_L}$ in 
Eq.~(\ref{eqn:cs_multi_set_model_unify}) to 
$\max_{\bp\in \Delta_L}$.
\end{itemize}
Besides, 
to the best of our knowledge, this is the first time 
that such a formulation 
Eq.~(\ref{eqn:cs_multi_set_model_unify}) is proposed. 
In the following, we will focus on the computational methods. 
Note that the difficulties in solving 
Eq.~(\ref{eqn:cs_multi_set_model_unify})
are due to two aspects: 
%=================================
\begin{itemize}
\item Optimization over $\bp$:
Although classical methods to minimize over $\bp$ with fixed $\bx$, 
e.g., \emph{alternative minimization} and ADMM \cite{boyd2011alternating},
can calculate local minimum efficiently (due to the 
bi-convexity of Eq.~(\ref{eqn:cs_multi_set_model_unify}), they 
can be easily trapped in the local-minima. 
This is because some entries 
in $\bp$ can be set to zero and hence $\bx$ 
will be kept away from the corresponding set $\mathcal{C}_i$
thereafter.
To handle this problem, we propose to 
use \emph{multiplicative weight update} \cite{arora2012multiplicative} and 
update $\bp$ with the relation 
$\bp^{(t+1)} \propto \bp^{(t)}e^{-\eta^{(t)}_p f_i(\bx)}$, where 
$\bp^{(t)}$ denotes $\bp$'s value in the $t$th iteration. 
This update relation avoids the sudden change of $\bp^{(t)}$'s
entries from non-zero to zero, which could have forced 
$\bx^{(t)}$ being trapped in a local minimum. 
%=================================
\item 
Optimization over $\bx$:  
Due to the non-smoothness of $\wt{\Ind}(\bx\in \calC_i)$ and
$\|\bx\|_1$ in 
Eq.~(\ref{eqn:cs_multi_set_model_unify}) 
and the difficulties in calculating their sub-gradients, 
directly minimizing 
Eq.~(\ref{eqn:cs_multi_set_model_unify}) 
would be computationally prohibitive.
We propose to first approximate 
$\wt{\Ind}(\bx\in \calC_i)$ with a smooth function $h_i(\bx)$ 
and update $\bx^{(t)}$ with the 
relation Eq.~(\ref{eqn:non_cvx_prox_x_update}) used in 
\emph{proximal gradient descent} 
\cite{beck2017first}.

\end{itemize}

\begin{definition}[$L_g$-strongly smooth \cite{beck2017first}]
Function $g(\cdot):\mathcal{X}\mapsto \RR$ is $L_g$-strongly smooth 
iff 
\[
g(\by) \leq g(\bx) + \la \nabla g(\bx), ~\by-\bx\ra + \
\dfrac{L_g}{2}\|\bx - \by\|^2_2,  
\]
for all $\bx,~\by$ in the domain $\mathcal{X}$.
\end{definition}

\subsection{Non-convex Proximal Multiplicative Weighting Algorithm}
%===========================
Here we directly approximate the truncated indicator function 
$\wt{\Ind}(\bx\in \calC_i)$
by $L_{h, i}$ strongly-smooth convex penalty functions $h_i(\bx)$, 
which may be different for different shapes of convex sets. 
For example, consider the convex set $\calC_i$ in 
Example.~\ref{example:support_set}.
We may define 
$
h_i(\bx) = \sum_{j\notin [i, i+K]}^{N}x_j^2, 
$ 
where $[a, b]$ denotes the region from $a$ to $b$.  
While for the set $\{\bx:\la \ba, \bx\ra \leq b\}$, 
we may instead adopt the modified log-barrier function
with a finite value. 
Then Eq.~(\ref{eqn:cs_multi_set_model_unify})
can be rewritten as 
\begin{equation}
\label{eqn:cs_multl_set_model_unify_smooth}
\min_{\bp}\min_{\bx}~\
~\mathcal{L}(\bp,\bx) \
\triangleq 
\sum_{i=1}^L p_i f_i(\bx),
\end{equation}
where $f_i(\bx)$ is defined as 

\[
f_i(\bx) \defequal
\|\bx\|_1 +  h_i(\bx) + \
\dfrac{\lambda_1}{2}\|\by - \bA\bx\|_2^2 + \
\dfrac{\lambda_2}{2}\|\bx\|_2^2.
\]
Hence, the optimization problem in 
(\ref{eqn:cs_multl_set_model_unify_smooth}) 
can be solved via Alg.~\ref{alg:non_cvx_prox_mw}

\begin{lemma}
$h(\bx)\defequal \sum_i p_i h_i(\bx) + \
\frac{\lambda_1\|\by - \bA\bx\|^2_2}{2} + \
\frac{\lambda_2 \|\bx\|^2_2}{2}$ is 
strongly-smooth with some positive constant denoted as $L_h$. 
\end{lemma}

\begin{proof}
First, we can check that 
$\frac{\lambda_1\|\by - \bA\bx\|^2_2}{2} + \
\frac{\lambda_2 \|\bx\|^2_2}{2}$ is strongly-smooth. 
Denote the corresponding parameter as $L_{h, 0}$. 
Meanwhile, due to the construction of $h_i(\bx)$, 
it is strongly-smooth for every $i$. Since $p_i$ is non-negative 
for every $i$,  
we can easily prove the following 
inequality
\[
h(\bx_1) \geq h(\bx_2) + \
\la \nabla h(\bx_2), \bx_1 - \bx_2\ra + \
\dfrac{L_h}{2}\|\bx_1 -\bx_2\|^2_2, 
\]
where $L_h$ is defined as $\min(L_{h, i})$, $0\leq i \leq L$.
\end{proof}
Then we have the following theorem.
%======================
\par

\begin{restatable}{theorem}{noncvxproxlocalconvergethm}
\label{thm:noncvxproxlocalconverge} 
Let $\eta_x^{(t)} = \eta_x \leq L_h^{-1}$, 
and $\eta^{(t)}_p = R_f^{-1}\sqrt{2\log L/T}$,
where $|f_i(\cdot)|\leq R_f$, 
$\|\bx\|_2 \leq R$. Then 
we have 
\[
\begin{aligned}
&\bigg|\dfrac{\min_{\bp}\sum_t \calL(\bp, \bx^{(t)})}{T} - \dfrac{\sum_t \calL(\bp^{(t)}, \bx^{(t)})}{T}\bigg| \\
+~&
\bigg| \dfrac{\sum_t \calL(\bp^{(t)}, \bx^{(t)})}{T}- \
\dfrac{\min_{\bx}\sum_t \calL(\bp^{(t)}, \bx)}{T}\bigg|\\
\leq~&\frac{2R^2}{\eta_x T} + R_f\sqrt{\frac{\log L}{T}},
\end{aligned}
\]
where $T$ denotes the number of iterations. 
\end{restatable}
Due to the difficulties in analyzing the global 
optimum, 
in Theorem \ref{thm:noncvxproxlocalconverge} 
we focus on analyzing the closeness between the
average value $\frac{\sum_t \calL(\bp^{(t)}, \bx^{(t)})}{T}$
to its local minimum. 
The first term denotes the gap 
between average value $\frac{\sum_t \calL(\bp^{(t)}, \bx^{(t)})}{T}$
and the optimal value of $\calL(\bp, \bx)$ 
with $\bx^{(t)}$ being fixed. Similarly, 
the second term represents the gap 
with $\bp^{(t)}$ being fixed. 
As $T\rightarrow\infty$, the sum of these two bounds
approaches to zero at the rate of 
$\calO(T^{-1/2})$. 

\par 
Moreover note that setting $\eta^{(t)}_p$ requires the 
oracle knowledge of $T$, which is impractical. 
This artifacts can easily be fixed by the doubling 
trick \cite[\S 2.3.1]{shalev2012online}. 
%Due to the space limit, we omit the corresponding discussion. 
In addition, we have proved the following theorem. 
\vsp 
%+==============================================
\begin{restatable}{theorem}{noncvxproxsmoothgradthm}
\label{thm:noncvxproxsmoothgrad}
Let $\eta^{(t)}_w \leq L_h^{-1}$,
where  $|f_i(\cdot)|\leq R_f$. Then
we have 
\[
\dfrac{1}{T}\sum_t \|\bx^{(t+1)} - \bx^{(t)}\|^2_2 \leq \
\dfrac{2\calL(\bp^{(0)}, \bx^{(0)})}{L_h T} + \
\dfrac{4 R_f^2 \sum_t \eta^{(t)}_p }{L_h T}. 
\]
\end{restatable}

This theorem discusses the convergence speed with respect to
the $\bx^{(t)}$ update. Due to the $\calO(T^{-1})$ of the 
first term on the right side of the above inequality, 
the best convergence rate we can obtain is 
$\calO(T^{-1})$, which is achievable by 
$\eta_p^{(t)} \propto t^{-2}$. However, using fixed learning rate $\eta_p$ 
as in Thm.~\ref{thm:noncvxproxlocalconverge}
would result in the convergence rate of $\calO(T^{-1/2})$.

\subsection{Regularization for $\mathbf{p}$}
%===============================================
\label{subsec:regular_p}
Another drawback of the naive method is that
they cannot exploit the prior knowledge. For example, 
if we know that the true $\bx^{\natural}$ is most likely to reside in set $\calC_1$.
With the naive method, 
we cannot use this information but separately solve 
Eq.~(\ref{eqn:cs_mutli_set_optim}) for all $L$ sets. 
In the sequel, we will show that 
our formulation Eq.~(\ref{eqn:cs_multi_set_model_unify})
can incorporate such prior knowledge by adding 
regularizers for $\bp$, and bring certain 
performance improvement.
\par 
Note that we can interpret $p_i$, the $i$-th element of $\bp$  
in Eq.~(\ref{eqn:cs_multl_set_model_unify_smooth}) as 
the likelihood of $\bx^{\natural}\in \calC_i$. 
Without any prior knowledge about
which set $\calC_i$ the true signal $\bx^{\natural}$ resides, 
variable $\bp$ is uniformly distributed among
all possible distributions $\Delta_L$.
When certain prior information is available, its 
distribution is skewed towards certain distributions, 
namely $\bq$. 

\par 
In this paper, we adopt $\|\cdot\|^2_2$ to regularize
$\bp$ towards $\bq$ and write the modified function 
$\+{LR}(\bp, \bx)$ as
\[
\+{LR}(\bp, \bx) = \calL(\bp, \bx) + \
\dfrac{\lambda_3}{2}\|\bp - \bq\|^2_2,
\]
where $\lambda_3 > 0$ is a constant used to 
balance $\calL(\bp, \bx)$ and $\frac{1}{2}\|\bp - \bq\|^2_2$. 
Based on different applications, other norms such as KL-divergence or
$l_1$ norm can be used as the regularizer. 
\par
Then we substitute the update equation Eq.~(\ref{eqn:non_cvx_prox_p_update})
as 

\[
\bp^{(t+1)} = \Proj_{\Delta}\left(\bp^{(t)} - \eta^{(t)}_p \bg^{(t)}\right), 
\]
where 
$\bg^{(t)} = \nabla_{\bp^{(t)}}\+{LR}(\bp, \bx^{(t)})=\bff(\bx^{(t )}) + \lambda_3(\bp^{(t)} - \bq)$, 
and
$\bff(\bx^{(t)})$ denotes the vector whose 
$i$th element is $f_i(\bx^{(t)})$. 
Similar as above, we obtain the following theorems. 

\par \vsp 
\begin{restatable}{theorem}{noncvxproxregularsaddlethm}
\label{thm:noncvxproxregularsaddle}
Provided that $\|\bg^{(t)}\|_2 \leq R_g$, 
by setting $\eta_x^{(t)} = \eta_x \leq L_h^{-1}$
and $\eta_p^{(t)} = (\lambda t)^{-1}$ 
we conclude that

\[
\begin{aligned}
&\bigg|\dfrac{\min_{\bp}\sum_t \+{LR}(\bp, \bx^{(t)})}{T} - \dfrac{\sum_t \+{LR}(\bp^{(t)}, \bx^{(t)})}{T}\bigg| \\
+ ~& \bigg| \dfrac{\sum_t \+{LR}(\bp^{(t)}, \bx^{(t)})}{T}- \
\dfrac{\min_{\bx}\sum_t \+{LR}(\bp^{(t)}, \bx)}{T}\bigg| \\
\leq~&\dfrac{ R_g^2 \log T}{2\lambda_3 T} + \dfrac{R^2}{2\eta_x T},
\end{aligned}
\]	

\end{restatable}

Comparing with Thm.~\ref{thm:noncvxproxlocalconverge}, 
Thm.~\ref{thm:noncvxproxregularsaddle} implies that
the regularizers improve the optimal 
rate from $\calO(T^{-1/2})$ to $\calO(\log T/T)$.
Therefore, our framework can exploit the prior information 
to improve the recovery performance whereas the naive
method of iterative computation fails to achieve as such.

%+=================================
\par  \vsp 
\begin{restatable}{theorem}{noncvxproxregularconvergethm}
\label{thm:noncvxproxregularconverge}
Provided that $\|\bg^{(t)}\|_2 \leq R_g$, by setting $\eta^{(t)}_w = \eta_x \leq L_h^{-1}$ we conclude that 
\[
\begin{aligned}
\dfrac{1}{T}\sum_t \|\bx^{(t+1)} - \bx^{(t)}\|^2_2 
\leq \frac{2\+{LR}(\bp^{(0)}, \bx^{(0)})}{L_h T} + \dfrac{2  R_g^2}{L_h T}\
\sum_{t}\left(\eta^{(t)}_p + \frac{\lambda_3 \left(\eta^{(t)}_p\right)^2}{2}\right).
\end{aligned} 
\]	
\end{restatable}
In this case, if we set $\eta^{(t)}_p$ as $t^{-2}$, 
then $\frac{1}{T}\sum_t \|\bx^{(t+1)} - \bx^{(t)}\|^2_2$
would decrease at the rate of $\calO(T^{-1})$, which is the 
same as Thm.~\ref{thm:noncvxproxsmoothgrad}.

\section{Conclusion}
%============================
In this paper, we studied the 
compressive sensing with a multiple convex-set domain. 
First we analyzed the impact of prior knowledge 
$\bx\in \bigcup_i \calC_i$ on the minimum number of measurements $M$
to guarantee uniqueness of the solution. 
We gave an illustrative example and 
showed that significant savings in $M$ can be achieved. 
Then we formulated a universal objective
function and develop an algorithm for the signal reconstruction.
We show that in terms of the speed of convergence
to local minimum,
our proposed algorithm based on \emph{multiplicative weight update}
and \emph{proximal gradient descent} can achieve the optimal 
rate  of $\calO(T^{-1/2})$. 
Further, in terms of
$T^{-1}\|\bx^{(t+1)} - \bx^{(t)}\|^2_2$, the optimal speed 
 increases to $\calO(T^{-1})$. 
 Moreover, provided that we have a prior knowledge about $\bp$,
 we show that we can improve 
the optimal recovery performance 
by $\|\cdot\|_2^2$ regularizers, and hence increasing
the above convergence rate from 
$\calO(T^{-1/2})$ to $\calO\bracket{\frac{\log T}{T}}$.

\bibliographystyle{authordate3}
\bibliography{multi_set}

\begin{appendices}

\section{Proof of \autoref{thm:cs_multi_set-uniq}}
\label{thm_proof:cs_multi_set-uniq}

\begin{proof}
Note that for any-vector 
non-zero $\V{h}\in \Null(\bA)\bigcap \calT\bigcap \left(\bigcup_{i,j} \wt{\calC}_{i,j}\right)$, 
we can always rescale to make it unit-norm. 
Hence we can rewrite the event $\calE$ as 
\[
\calE = \
\left\{\Null(\bA)\bigcap \SS^{n-1}_2 \bigcap \
\calT \bigcap \left(\bigcup_{i,j} \wt{\calC}_{i,j}\right) = \emptyset \right\}.  
\]
For the conciseness of notation, we define 
$\wt{\calC}$ to be $\wt{\calC} = \bigcup_{i,j} \wt{\calC}_{i,j}$. 
Then we upper-bound $1- \Prob(\calE)$ as 
\[
\begin{aligned}
& 1 - \Prob(\calE) = \
\Prob\bracket{\Null(\bA)\bigcap \SS^{n-1}_2 \bigcap \
\calT \bigcap \wt{\calC}\neq \emptyset } \\
\stackrel{(i)}{\leq} &
\underbrace{\Prob\bracket{\Null(\bA)\bigcap \SS^{n-1}_2 \bigcap \
\calT  \neq \emptyset }}_{\calP_1}\vcap
\underbrace{\Prob\bracket{\Null(\bA)\bigcap \SS^{n-1}_2 \bigcap \
\wt{\calC} \neq \emptyset }}_{\calP_2}, 
\end{aligned}
\]
where $(i)$ is because 
\[
\begin{aligned}
&\left\{\Null(\bA)\bigcap \SS^{n-1}_2 \bigcap \
\calT \bigcap \wt{\calC}\neq \emptyset \right\} \
\subseteq \left\{\Null(\bA)\bigcap \SS^{n-1}_2 \bigcap \
\calT  \neq \emptyset \right\},\\
&\left\{\Null(\bA)\bigcap \SS^{n-1}_2 \bigcap \
\calT \bigcap \wt{\calC}\neq \emptyset \right\} \
\subseteq \
\left\{\Null(\bA)\bigcap \SS^{n-1}_2 \bigcap \
\wt{\calC} \neq \emptyset \right\}
\end{aligned}.
\]
With \autoref{lemma:uniquesolp1bound} and 
\autoref{lemma:uniquesolp2bound}, we can 
separately bound $\calP_1$ and $\calP_2$ and 
finish the proof.
	
\end{proof}

%========================================
\begin{lemma}
\label{lemma:uniquesolp1bound}
We have 
$\calP_1  \leq 1 \vcap \exp\left(-\frac{\left(a_m -\omega(\calT)\right)^2}{2}\right)$, if $a_m \geq \omega(\calT)$.
\end{lemma}

\begin{proof}
Note that we have 
\[
\begin{aligned}
&\underbrace{\Prob\bracket{\Null(\bA)\bigcap \SS_2^{n-1}\bigcap \
\calT \neq \emptyset }}_{\calP_1} \
+ \underbrace{\Prob\bracket{\Null(\bA)\bigcap \SS_2^{n-1}\bigcap \
\calT = \emptyset}}_{\calP_1^c}= 1.
\end{aligned}
\]
Then we lower-bound $\calP_1^c$ as 
\[
\begin{aligned}
&\calP_1^c = \Prob\bracket{\min_{\bu \in \SS^{n-1}_2 \bigcap \calT}\|\bA\bu\|_2 > 0} \
\stackrel{(i)}{\geq}  1 - \exp\left(-\frac{\left(a_m -\omega\left(\calT\right)\right)^2}{2}\right), 
\end{aligned}
\]
provided $a_m \geq \omega(\calT)$, 
where $(i)$ is because of Corollary 3.3 in 
\cite{chandrasekaran2012convex}, 
and $\omega(\cdot)$ denotes the Gaussian width.	
\end{proof}

%========================================
\begin{lemma}
\label{lemma:uniquesolp2bound}
If $ \omega(\wt{\calC}_{ij}) \leq 1 - 2\epsilon a_M$,
we have 
\[
\begin{aligned}
\calP_2 \leq 1 \vcap \frac{3}{2}\exp\left(-\frac{\epsilon ^2a_M^2}{2}\right) \
+ \
\sum_{i \leq j} \exp\bracket{-\dfrac{\bracket{(1-2\epsilon)a_M -\
\omega(\wt{\calC}_{ij})}^2}{2}},
\end{aligned}
\] 
\end{lemma}

%========================================
\begin{proof}
Note that we have 
\[
\begin{aligned}
&\underbrace{\Prob\bracket{\Null(\bA) \bigcap {S}^{n-1}_2 \bigcap \wt{\calC} \neq  \emptyset }}_{\calP_2}\
 + \
\underbrace{\Prob\bracket{\Null(\bA) \bigcap {S}^{n-1}_2 \bigcap \wt{\calC} = \emptyset }}_{\calP_2^c}\
= 1. 
\end{aligned}
\]
Here we upper-bound $\calP_2$ via lower-bounding $\calP_2^c$. 
First we define $\calP_2^c(d)$ as 
\[
\calP_2^c(d) \defequal \
\Prob\bracket{\min_{\bu\in \bigcup \+S_{i,j},~\bv\in \Null(\bA)}\|
\bu - \bv\|_2 \geq d}. 
\]
Then we have $\calP_2^c = \lim_{d\rightarrow 0} \calP_2^c(d)$.
The following proof trick is 
fundamentally the 
same as that are used in Theorem 4.1 in \cite{gordon1988milman}
but in a clear format by only keeping the necessary parts
for this scenario. 
We only present it for the 
self-containing of this paper and do not 
claim any novelties.  
\par 
We first 
define $\+S_{i,j} = \SS^{n-1}_2 \bigcap \wt{\calC}_{i,j}$
and two quantities 
$Q_1$ and $Q_2$ as 
\[
\begin{aligned}
Q_1 &\defequal \Prob\
\bracket{\min_{\bu\in \bigcup \+S_{i,j}}\|\bA\bu\|_2 \geq d(1+\epsilon_1)\Expc\|\bA\|_2}, \\
Q_2 &\defequal \Prob\
\bracket{\bigcap_{i_1\leq j_1}\bigcap_{\bu\in \+S_{i_1,j_1}}\
\bracket{\left(\sum_{i_2=1}^M g_{i_2}^2\right)^{1/2} + \sum_{j_2} u_{j_2} h_{j_2} \geq \
d(1+\epsilon_1)\Expc\|\bA\|_2 + \epsilon_2 a_M }}, 
\end{aligned}
\]
where $a_M = \Expc \|\bg\|_2,~\bg\in \normdist(\bZero,~\bI_{M\times M})$, 
and $g_j$, $h_i$ are iid standard normal random variables
$\normdist(0, 1)$. The following proof 
is divided into $3$ parts. 
\par 
\textbf{Step I}. We prove that 
$\calP^c_2(d) + e^{-\epsilon_1^2 a_M^2/2} \geq Q_1$, which is 
done by 
\[
\begin{aligned}
Q_1 &= \Prob\bracket{\min_{\bu\in \bigcup \+S_{i,j}}\|\bA\bu\|_2 \geq \
d(1+\epsilon_1)\Expc\|\bA\|_2} \\
=~& \
\Prob\bracket{\min_{\bu\in \bigcup \+S_{i,j},~\bv\in \Null(\bA)}\
\|\bA(\bu-\bv)\|_2 \geq d(1 + \epsilon_1)\Expc \|\bA\|_2}\\
\stackrel{(i)}{\leq}~& \
\Prob\bracket{\min_{\bu\in \bigcup \+S_{i,j},~\bv\in \Null(\bA)}\
\|\bA\|_2 \|\bu - \bv\|_2 \geq d(1 + \epsilon_1)\Expc \|\bA\|_2}\\
\stackrel{(ii)}{\leq} ~& \
\Prob\bracket{\|\bA\|_2 \geq (1+\epsilon_1 )\Expc\|\bA\|_2} \
+ \underbrace{\Prob\bracket{\min_{\bu\in \bigcup\+S_{i,j},~\bv\in \Null(\bA)}~\|\bu - \bv\|_2 \geq d}}_{\leq~\calP_2^c(d)} \\
\stackrel{(iii)}{\leq} ~& \
\exp\left(-\dfrac{\epsilon_1^2(\Expc \|\bA\|_2)^2}{2}\right) + \
\calP_2^c(d) \\
\stackrel{(iv)}{\leq}& \exp\left(-\dfrac{\epsilon_1^2 a_M^2}{2}\right) + \calP_2^c(d),
\end{aligned}
\]
where in $(i)$ we use 
$\|\bA\|_2 \|\bu - \bv\|_2 \geq \|\bA(\bu-\bv)\|_2$, 
in $(ii)$ we use the union bound for 
\[
\begin{aligned}
&\left\{\min_{\bu\in \bigcup \+S_{i,j},~\bv\in \Null(\bA)}\
\|\bA\|_2 \|\bu - \bv\|_2 \geq d(1 + \epsilon_1)\Expc \|\bA\|_2\right\} \\
\subseteq ~& \
\left\{\|\bA\|_2 \geq (1+\epsilon_1 )\Expc\|\bA\|_2\right\} \bigcup \
\left\{\min_{\bu\in \bigcup\+S_{i,j},~\bv\in \Null(\bA)}~\|\bu - \bv\|_2 \geq d \right\},
\end{aligned}
\]
in $(iii)$ we use the Gaussian 
concentration inequality Lipschitz functions
(Theorem 5.6 in \cite{boucheron2013concentration}) for $\|\bA\|_2$, 
and in $(v)$ we use 
$\|\bA\|_2 \geq \|\bA\be_1\|_2 = \|\sum_{i=1}^M A_{i, 1}\|_2$, 
where $\be_1$ denotes the canonical basis.
\par 
\textbf{Step II}. We prove that 
$Q_1 + \frac{1}{2}e^{-\epsilon_2^2a_M^2/2} \geq Q_2$, 
which is done by 
\[
\begin{aligned}
& Q_1 + \frac{1}{2}e^{-\epsilon^2a_M^2/2} \stackrel{(i)}{\geq} \
Q_1 + \Prob\{g \geq \epsilon_2 a_M \} \\
\stackrel{(ii)}{\geq} ~& \Prob\bracket{\
\min_{\bu\in \bigcup \+S_{i_1,j_1}}\|\bA\bu\|_2 + g\|\bu\|_2 \geq \
d(1 + \epsilon)\Expc\|\bA\|_2 + \
\epsilon_2 a_M \|\bu\|_2 } \\
=~&\Prob\bracket{\bigcap_{i_1\leq j_1}\bigcap_{\bu\in \+S_{i_1,j_1}}\
\|\bA\bu\|_2 + g\|\bu\|_2 \geq \
d(1 + \epsilon)\Expc\|\bA\|_2 + \
\epsilon_2 a_M \|\bu\|_2} \\
\stackrel{(iii)}{\geq}~&\
\underbrace{\Prob\bracket{\
\bigcap_{i_1\leq j_1}\bigcap_{u\in \+S_{i_1,j_1}} \
\left(\sum_{i_2 = 1}^M g_{i_2}^2\right)^{\frac{1}{2}} + \
\sum_{j_2 = 1}^N u_{j_2}h_{j_2}\geq \
d(1 + \epsilon_1)\Expc\|\bA\|_2 + \
\epsilon_2 a_M }}_{Q_2}, 
\end{aligned}
\]
where in 
$(i)$ $\bg$ is a RV satisfying standard normal distribution, 
in  $(ii)$ we use the union bound, 
and $(iii)$ comes from Lemma 3.1 in \cite{gordon1988milman}
and $\|\bu\|_2 = 1$. 
\par 
\textbf{Step III}. We lower bound $Q_2$ as 
\[
\begin{aligned}
& 1 - Q_2 =  \Prob\bracket{\bigcup_{i_1 \leq j_1} \bigcup_{u\in \+S_{i_1, j_1}} \
~\left[\left(\sum_{i_2=1}^M g_{i_2}^2\right)^{\frac{1}{2}} + \
\sum_{j_2 = 1}^N u_{j_2}h_{j_2} \leq \
d(1+\epsilon_1)\Expc\|\bA\|_2 + \epsilon_2 a_M  \right]} \\
\leq ~&\Prob\bracket{\left(\sum_{i_2=1}^M g_{i_2}^2\right)^{\frac{1}{2}} \leq (1-\epsilon_2)a_M } + \
\Prob\bracket{\bigcup_{i_1 \leq j_1} \bigcup_{u\in \+S_{i_1, j_1}} \sum_{j_2 = 1}^N u_{j_2}h_{j_2}\leq \
d(1+\epsilon_1) \Expc \|\bA\|_2 - (1-2\epsilon_2)a_M } \\
\leq~& 
\Prob\bracket{\left(\sum_{i_2=1}^M g_{i_2}^2\right)^{\frac{1}{2}} -\
a_M \leq -\epsilon_2 a_M } + \
\Prob\bracket{\bigcup_{i_1 \leq j_1} \bigcup_{u\in \+S_{i_1, j_1}} \sum_{j_2 = 1}^N u_{j_2}h_{j_2}\leq \
d(1+\epsilon_1) \Expc \|\bA\|_2 - (1-2\epsilon_2)a_M } \\
\stackrel{(i)}{\leq} ~&  \exp\left(-\dfrac{\epsilon_2^2 a_M^2}{2}\right) + 
\Prob\bracket{\bigcup_{i_1 \leq j_1} \bigcup_{u\in \+S_{i_1, j_1}} \sum_{j_2 = 1}^N u_{j_2}h_{j_2}\leq \
d(1+\epsilon_1) \Expc \|\bA\|_2 - (1-2\epsilon_2)a_M } \\
\stackrel{(ii)}{\leq} ~&\exp\left(-\dfrac{\epsilon_2^2 a_M^2}{2}\right) + 
\sum_{i_1 \leq j_1} \Prob\bracket{\bigcup_{u\in \+S_{i_1, j_1}} \
\sum_{j_2 = 1}^N u_{j_2} h_{j_2} \leq \
d(1+\epsilon_1)\Expc \|\bA\|_2 - (1-2\epsilon_2)a_M } \\
\stackrel{(iii)}{\leq}~& \exp\left(-\dfrac{\epsilon_2^2 a_M^2}{2}\right) + \
\sum_{i_1 \leq j_1} \Prob\bracket{\max_{u \in \+S_{i_1, j_1}} \sum_{j_2 = 1}^N u_{j_2} h^{'}_{j_2} \geq \
(1-2\epsilon_2)a_M  - d(1+\epsilon_1) \Expc \|\bA\|_2} \\ 
\stackrel{(iv)}{\leq}~& \exp\left(-\dfrac{\epsilon_2^2 a_M^2}{2}\right) + \
\sum_{i_1 \leq j_1}  \exp\bracket{-\dfrac{\left((1-2\epsilon_2)a_M - d(1+\epsilon_1)\Expc\|\bA\|_2 - \omega(\wt{\calC}_{ij}) \right)^2}{2}}, 
\end{aligned} 
\]
where in $(i)$ we use 
$\Expc \sqrt{\sum_{i_2=1}^M g_{i_2}^2} = a_M$ and 
Gaussian concentration inequality in \cite{boucheron2013concentration}, 
in $(ii)$ we use union-bound, 
in $(iii)$ we define
$h^{'} = - h$ and flip the sign 
by the symmetry of Gaussian variables, 
and in $(iv)$ we 
use the definition of $\omega(\wt{\calC}_{ij})$.
Assuming 
$(1-2\epsilon_2)a_M \geq d(1+\epsilon_1)\Expc\|\bA\|_2 + \omega(\wt{\calC}_{ij})$, 
we finish the proof via 
the Gaussian concentration inequality 
\cite{boucheron2013concentration}. 
\par 
Combining the above together and set $d(1+\epsilon_1)\rightarrow 0$ while
$\epsilon_1 \rightarrow \infty$,  
we conclude that 
\[
\calP^c_2 \geq 1 - \frac{3}{2}\exp\left(-\frac{\epsilon ^2a_M^2}{2}\right) - \
\sum_{i \leq j} \exp\left(-\frac{\bracket{(1-2\epsilon)a_M -\
\omega(\wt{\calC}_{ij})}^2}{2}\right),
\]
provided $(1-2\epsilon)a_M \geq \omega(\wt{\calC}_{ij})$,
and finish the proof. 
	
\end{proof}

\section{Proof of \autoref{thm:noncvxproxlocalconverge}}
%==================================
\label{thm_proof:noncvxproxlocalconverge}

\begin{proof}
Define $\bp^{*}$ and $\bx^{*}$ as 
\begin{equation}
\label{eqn:p_x_star}
\bp^{*} = \argmin_{\bp}\sum_t \calL\left(\bp, \bx^{(t)}\right),~~~
\bx^{*} = \argmin_{\bx}\sum_t \calL\left(\bp^{(t)}, \bx\right).
\end{equation}
respectively
First we define $\calT^t_1$ and $\calT^t_2$ as
\[
\begin{aligned}
\calT^t_1 &= \calL(\bp^{(t)}, \bx^{(t)}) -  \calL(\bp^{*}, \bx^{(t)}); \\
\calT^t_2 &= \calL(\bp^{(t)}, \bx^{(t)}) - \calL(\bp^{(t)}, \bx^{*}),	
\end{aligned}
\]
respectively. Then our goal becomes 
bounding $|\sum_t \calT^t_1| + |\sum_t \calT^t_2|$.
With \autoref{lemma:noncvx_prox_T1bound} and 
\autoref{lemma:noncvx_prox_T2bound}, we have finished
the proof. 
\end{proof}

%===========================================
\begin{lemma}
\label{lemma:noncvx_prox_T1bound}
Define 
$\calT^t_1 = \calL(\bp^{(t)}, \bx^{(t)}) -  \calL(\bp^{*}, \bx^{(t)})$, 
where  $\bp^{*}$ is defined in Eq.~(\ref{eqn:p_x_star}), 
then we have 
\[
0 < \sum_t  \calT_1^t \leq R_f \sqrt{T\log L}, 
\]
when $\eta^{(t)}_p = R_f\sqrt{2\log L/T}$. 
\end{lemma}

\begin{proof}
Based on the definition of $\bp^{*}$, 
we note that $\sum_t \calT_1^t$ is non-negative and prove 
the lower-bound. 
Then we prove its upper-bound.
\par 
Since the function is linear, 
optimal $\bp^{*}$ must be at the edge of $\Delta_L$ and
we denote the non-zero entry as $i^{*}$. 
Hence, we could study it via the 
\emph{multiplicative weight algorithm} analysis \cite{arora2012multiplicative}.
First we rewrite the update equation (\ref{eqn:non_cvx_prox_p_update}).
Define
$\bw^{(0)} = \V{1} \in \RR^{L}$ and update
$\bw^{(t+1)}$ as 
\[
{w}^{(t + 1)}_i = w^{(t)}_i \exp\left(-\eta_p^{(t)}f_i(\bx^{(t)})\right),\
p^{(t+1)}_i = \dfrac{{w}^{(t + 1)}_i}{\sum_i {w}^{(t + 1)}_i}.
\]
where $(\cdot)_i$ denotes the $i$th element, and 
$\bp^{(t+1)}$ can be regarded as the normalized version of
$\bw^{(t+1)}$.
\par 
First we define $\Psi_t$ as 
\[
\Psi_t = \sum_{i=1}^L w^{(t)}_i. 
\]
Then we have $\Psi_0 = L$ while 
$\Psi_T \geq w^{(T)}_{i^{*}} =\
\exp\left(-\sum_{t=1}^T \eta^{(t)}_p f_{i^{*}}(\bx^{(t)})\right) = \
\exp\left( -\eta_p \sum_{t=1}^T f_{i^{*}}(\bx^{(t)})\right)$, where 
$\eta_p^{(t)} = \eta_p = \sqrt{2\log L/T}$. 
\par 
Then we study the
division $\Psi_{t+1}/\Psi_t$ as 
\[
\begin{aligned}
& \Psi_{t+1} = \sum_{i}w^{(t+1)}_i = \
\sum_i w^t_i \exp\left(-\eta_p \V{f}_i(\bx^{(t)})\right) \\
\stackrel{(i)}{\leq} ~& \
\sum_i w^t_i \left(1 - \eta_p \V{f}_i(\bx^{(t)}) + \frac{\eta_p^2 f^2_i(\bx^{(t)})}{2} \right) \\
= ~&\Psi_t\left(1 - \eta_p \la \bp^{(t)}, \V{f}(\bx^{(t)}) \ra + \frac{\eta_p^2 R_f^2}{2} \right) \\
\stackrel{(ii)}{\leq} ~& \
\Psi_t \exp\left(-\eta_p \la \bp^{(t)}, \V{f}(\bx^{(t)}) \ra  + \
\frac{\eta_p^2 R_f^2}{2}\right)
\end{aligned}
\]
where in $(i)$ we use $e^{-x}\leq 1 + x + x^2/2$ for $x\geq 0$, 
and in $(ii)$ we use $e^x \geq 1 + x$ for all $x \in \RR$.
Using the above relation iteratively, we conclude that 
\[
\dfrac{\Psi_T}{\Psi_0} \leq \
\exp\left(-\eta_p\sum_t \la \bp^{(t)}, \V{f}(\bx^{(t)})\ra + \frac{\eta_p^2 T R_f^2}{2} \right), 
\]
which gives us 
\[
\log \Psi_T \leq \log L - \eta_p \sum_t \la \bp^{(t)}, \V{f}(\bx^{(t)}) \ra + \
\frac{T\eta_p^2 R_f^2}{2}. 
\]
With relation $\log \Psi_T \geq \log w_{i^{*}}^{(T)}$, we obtain 
\[
\eta_p \sum_t \left(\la \bp^{(t)}, \V{f}(\bx^{(t)})\ra - \la \be_{i^{*}}, \V{f}(\bx^{(t)})\ra \right) 
=\eta_p \sum_t \left(\la \bp^{(t)}, \V{f}(\bx^{(t)})\ra - \la \bp^{*}, \V{f}(\bx^{(t)})\ra \right)
\leq \log L + \frac{T\eta_p^2 R_f^2}{2}, 
\]
where $\be_{i^{*}}$ denotes the canonical basis, namely, 
has $1$ in its $i^{*}$th entry and all others to be zero.

\end{proof}

%===========================================
\begin{lemma}
\label{lemma:noncvx_prox_T2bound}
Define 
$\calT^t_2 = \calL(\bp^{(t)}, \bx^{(t)}) - \calL(\bp^{(t)}, \bx^{*})$, 
where $\bx^{*}$ is defined in Eq.~(\ref{eqn:p_x_star}), and 
set $\eta_x^{(t)} = \eta_x \leq L_h^{-1}$, then we
have
\[ 
0 \leq \sum_t \calT^{t}_{2} \leq \
\dfrac{1}{2\eta_x}\|\bx^{(0)} - \bx^{*}\|^2_2.
\]
\end{lemma}

\begin{proof}
From the definition of $\bx^{*}$, we can prove the 
non-negativeness of $\sum_t \calT^t_2$. 
Here we focus on upper-bounding $\sum_t \calT^t_2$ by 
separately analyzing each term $\calT^t_2$. For the conciseness
of notation, we drop the time index $t$. 
Define $h(\bx)$ as
\[
h(\bx) = \sum_i p_i h_i(\bx) + \
\frac{\lambda_1\|\by - \V{Ax}\|^2_2}{2} + \
\frac{\lambda_2 \|\bx\|^2_2}{2}.
\]
First, we rewrite the update equation as 
\[
\bx^{(t+1)} = \bx^{(t)} - \underbrace{\bracket{\bx^{(t)} - \prox_{\eta_x \|\cdot\|_1}\
\left[\bx^{(t)} - \eta_{w} \nabla h(\bx^{(t)})\right]}}_{\eta_x H(\bx^{(t)})},
\]
which means that 
\[
H(\bx^{(t)}) = \dfrac{\bx^{(t)} - \prox_{\eta_x \|\cdot\|_1}\
\bracket{\bx^{(t)} - \eta_{w} \nabla h(\bx^{(t)})}}{\eta_x}, 
\]
where $\prox_{\eta_x \|\cdot\|_1}(\bx)$ is defined as
\cite{beck2017first}
\[
\prox_{\eta_x \|\cdot\|_1}(\bx) = \argmin_{\bz}~\
\eta_x \|\bz\|_1 + \dfrac{1}{2}\|\bx - \bz\|^2_2.
\]
Here we need 
one important property of 
$H(\bx^{(t)})$, that is widely in the 
analysis of proximal gradient descent (direct results of 
Theorem 6.39 in \cite{beck2017first}) 
and states  
\[
H(\bx^{(t)}) \in \nabla h(\bx^{(t)}) + \partial \|\bx^{(t+1)}\|_1.
\]
Here we consider the update 
relation $f(\bx^{(t+1)}) - f(z)$ as 
\[
\begin{aligned}
& \calT^{t+1}_2 = \sum_{i} {p}^{(t)}_i \left[f_i(\bx^{(t+1)}) - f_i(\bx^{*}) \right] = \
\|\bx^{(t+1)}\|_1 + h(\bx^{(t+1)}) - \|\bx^{*}\|_1 - h(\bx^{*}) \\
\stackrel{(i)}{\leq} ~& \la \partial \|\bx^{(t+1)}\|_1,~\bx^{(t+1)} - \bx^{*} \ra + \
h(\bx^{(t)}) + \la \nabla h(\bx^{(t)}),~\bx^{(t+1)} - \bx^{(t)} \ra + \dfrac{L_h}{2}\|\bx^{(t+1)} - \bx^{(t)}\|^2_2 - h(\bx^{*})\\
\stackrel{(ii)}{\leq} ~&  \la \partial  \|\bx^{(t+1)}\|_1,~\bx^{(t+1)} - \bx^{*} \ra + \
\la \nabla  h(\bx^{(t)}),~\bx^{(t)} - \bx^{*} + \bx^{(t+1)} - \bx^{(t)} \ra + \dfrac{L_h}{2}\|\bx^{(t+1)} - \bx^{(t)}\|^2_2 \\
\stackrel{(iii)}{=}~& \
\la \partial  \|\bx^{(t+1)}\|_1 + \nabla  h(\bx^{(t)}),\
 \bx^{(t+1)} - \bx^{*} \ra + \dfrac{L_h}{2}\|\bx^{(t+1)} - \bx^{(t)}\|^2_2 \\
\stackrel{(iv)}{\leq} ~& \la H(\bx^{(t)}), ~\bx^{(t+1)} - \bx^{*}\ra +  \dfrac{L_h}{2}\|\bx^{(t+1)} - \bx^{(t)}\|^2_2 \\
\stackrel{(v)}{=}~& \frac{1}{\eta_x}\
\la \bx^{(t)} - \bx^{(t+1)}, \bx^{(t+1)} - \bx^{*}\ra + \dfrac{L_h}{2}\|\bx^{(t+1)} - \bx^{(t)}\|^2_2 \\
\stackrel{(vi)}{\leq}~&\frac{1}{\eta_x}\
\la \bx^{(t)} - \bx^{(t+1)}, \bx^{(t+1)} - \bx^{*} \ra + \dfrac{1}{2\eta_x} \|\bx^{(t+1)} - \bx^{(t)}\|^2_2\\
=~& \frac{1}{\eta_x}\
\la \bx^{(t)} - \bx^{(t+1)}, \bx^{(t+1)} - \bx^{*} \ra + \dfrac{1}{2\eta_x} \|\bx^{(t+1)} - \bx^{*}\|^2_2\ 
+ \dfrac{1}{2\eta_x} \|\bx^{(t)} - \bx^{*}\|^2_2 + \
\dfrac{1}{\eta_x} \la \bx^{(t+1)} - \bx^{*},~\bx^{*} - \bx^{(t)}\ra  \\
=~& \dfrac{1}{2\eta_x}\|\bx^{(t)} - \bx^{*}\|^2_2 + \dfrac{1}{2\eta_x} \|\bx^{(t+1)} - \bx^{*}\|^2_2 + \
\frac{1}{\eta_x}\la \bx^{(t+1)} - \bx^{*},~\bx^{*} - \bx^{(t)} + \bx^{(t)} - \bx^{(t+1)} \ra \\
=~&  \dfrac{1}{2\eta_x}\|\bx^{(t)} - \bx^{*}\|^2_2 - \
\dfrac{1}{2\eta_x}\|\bx^{(t+1)} - \bx^{*}\|^2_2,
\end{aligned}
\]
where in $(i)$ we use 
$\|\bx^{*}\|_1 \geq \|\bx^{(t+1)}\|_1 + \la \partial \|\bx^{(t+1)}\|_1, \bx^{*} - \bx^{(t+1)}\ra$ based on the definition of
sub-gradients, and 
$h(\bx^{t+1}) \leq h(\bx^{(t)}) + \
\la \nabla h(\bx^{(t)}),~ \bx^{(t+1)} \ra + L_h\|\bx^{(t+1)} - \bx^{(t)}\|^2_2/2 $ from the $L_h$ smoothness of $h(\cdot)$, 
in $(ii)$ we 
use $h(\bx^{*}) \geq h(\bx^{(t)}) + \la \nabla h(\bx^{(t)}),\bx^{*} -  \bx^{(t)} \ra$ since $h(\cdot)$ is convex, 
in $(iii)$ we use $H(\bx^{(t)}) \in \nabla h(\bx^{(t)}) + \partial \|\bx^{(t+1)}\|_1$, and 
in $(iv)$ we use $\bx^{(t+1)} = \bx^{(t)} - \eta_x H(\bx^{(t)})$, 
and in $(vi)$ we use $\eta_x \leq L^{-1}$.
\par 
Hence, we finishes the proof by
\[
\begin{aligned}
&\sum_t \calT^t_2 \leq \
\sum_t\left[\dfrac{1}{2\eta_x}\|\bx^{(t-1)} - \bx^{*}\|^2_2 - \
\dfrac{1}{2\eta_x}\|\bx^{(t)} - \bx^{*}\|^2_2\right]\\
=~& \dfrac{1}{2\eta_x}\|\bx^{(0)} - \bx^{*}\|^2_2 - \
\dfrac{1}{2\eta_x}\|\bx^{(T)} - \bx^{*}\|^2_2 \leq \
\dfrac{1}{2\eta_x}\|\bx^{(0)} - \bx^{*}\|^2_2 \stackrel{(i)}{\leq} \
\dfrac{4R^2}{2\eta_x}, 
\end{aligned}
\]
where in $(i)$ we use $\|\bx^{(0)} - \bx^{*}\|_2 \leq \
\|\bx^{*}\|_2 + \|\bx^{(0)}\| \leq 2R$.
\end{proof}

\section{Proof of \autoref{thm:noncvxproxsmoothgrad}}
\label{thm_proof:noncvxproxsmoothgrad}

\begin{proof}
%===========================================
First we define 
$h(\bx) = \sum_i p_i h_i(\bx) + \
{\lambda_1\|\by - \V{Ax}\|^2_2}/{2} + \
{\lambda_2 \|\bx\|^2_2}/{2}$.
Then we consider the term 
$\calL(\bp^{(t)}, \bx^{(t+1)}) - \calL(\bp^{(t)}, \bx^{(t)})$ and have 
\[
\begin{aligned}
& \calL(\bp^{(t)}, \bx^{(t+1)}) - \calL(\bp^{(t)}, \bx^{(t)}) = \
\|\bx^{(t+1)}\|_1 - \|\bx^{(t)}\|_1 + \sum_i p^{(t)}_i \
\left(h_i(\bx^{(t+1)}) - h_i(\bx^{(t)})\right) \\
\stackrel{(i)}{\leq}~&\
\la \partial \|\bx^{(t+1)}\|_1, \bx^{(t+1)} - \bx^{(t)}\ra + \
\sum_{i} p^{(t)}_i \left[\la\nabla h_i(\bx^{(t)}),~\bx^{(t+1)} - \bx^{(t)} \ra + \
\dfrac{L_h}{2}\|\bx^{(t+1)} - \bx^{(t)}\|^2_2 \right] \\
\stackrel{(ii)}{=}~&\
\la \underbrace{\partial \|\bx^{(t+1)}\|_1 + \sum_i p^{(t)}_i\nabla h_i(\bx^{(t)})}_{H(\bx^{(t)})},~\
\bx^{(t+1)} - \bx^{(t)}\ra + \
\dfrac{L_h}{2}\|\bx^{(t+1)} - \bx^{(t)}\|^2_2 \\
\stackrel{(iii)}{=}~&\
\la \frac{\bx^{(t)} - \bx^{(t+1)}}{\eta^{(t)}_w},~\bx^{(t+1)} - \bx^{(t)} \ra + \
\dfrac{L_h}{2}\|\bx^{(t+1)} - \bx^{(t)}\|^2_2 \\
=~& \dfrac{1}{2}\left(L_h - \frac{2}{\eta^{(t)}_w}\right)\|\bx^{(t+1)} - \bx^{(t)}\|^2_2 \\
\stackrel{(iv)}{\leq}~& -\dfrac{L_h}{2}\|\bx^{(t+1)} - \bx^{(t)}\|^2_2,
\end{aligned}
\]
where in $(i)$ we have 
$\|\bx^{(t)}\|_1 \geq \|\bx^{(t+1)}\|_1 + \la \partial \|\bx^{(t+1)}\|_1,~\bx^{(t)} -\bx^{(t+1)}\ra$ from the definition 
of sub-gradient \cite{beck2017first}, 
and $h_i(\bx^{(t+1)}) \leq h_i(\bx^{(t)}) + \la \nabla h_i(\bx^{(t)}), \bx^{(t+1)} - \bx^{(t)}\ra + \
\frac{L_h}{2}\|\bx^{(t+1)} - \bx^{(t)}\|^2_2$, 
in $(ii)$ we use the property 
$\partial \|\bx^{(t+1)}\|_1 + \sum_i p_i^{(t)} \nabla h_i(\bx^{(t)}) \in H(\bx^{(t)})$, 
in $(iii)$ we use 
$\bx^{(t+1)} = \bx^{(t)} - \eta^{(t)}_w H(\bx^{(t)})$, and 
in $(iv)$ we use $\eta^{(t)}_w \leq L^{-1}$.
\par 
%===========================================
Adopting similar tricks as \cite{qian2018robust}, 
we could upper-bound $\|\bx^{(t+1)} - \bx^{(t)}\|^2_2$ as 
\[
\begin{aligned}
&\|\bx^{(t+1)} - \bx^{(t)}\|^2_2 \leq \dfrac{2}{L_h}\
\left[\calL(\bp^{(t)}, \bx^{(t)}) - \calL(\bp^{(t)}, \bx^{(t+1)})\right] \\
=~& \
\dfrac{2}{L_h}\left[\underbrace{\calL(\bp^{(t)}, \bx^{(t)}) - \calL(\bp^{(t+1)}, \bx^{(t+1)})}_{\calT^t_1} \
+ \underbrace{\calL(\bp^{(t+1)}, \bx^{(t+1)}) - \calL(\bp^{(t)}, \bx^{(t+1)})}_{\calT^t_2}\right].
\end{aligned}
\]
Then we separately discuss bound $\calT^t_1$ and $\calT^t_2$.
Since most terms of $\sum_t \calT^t_1$ will be
cancelled after summarization, 
we focus the analysis on bounding $\calT^t_2$, which is 
\[
\begin{aligned}
& \calT^t_2 = \la \bp^{(t+1)} - \bp^{(t)}, \V{f}(\bx^{(t+1)})\ra \
\leq \
\|\bp^{(t+1)} - \bp^{(t)}\|_1 \underbrace{\|\V{f}(\bx^{(t+1)})\|_{\infty}}_{\leq R_f},
\end{aligned}
\]
where $\V{f}(\bx^{(t+1)})$ denotes the vector whose 
$i$th element is $f_i(\bx^{(t+1)})$.
Notice that we have
\[
\begin{aligned}
&\|\bp^{(t+1)} - \bp^{(t)}\|_1^2 \
\stackrel{(i)}{\leq} 2D_{KL}\left(\bp^{(t+1)}||\bp^{(t)}\right) \
\stackrel{(ii)}{\leq} \
2\eta^{(t)}_p\la \bp^{(t)} - \bp^{(t+1)},~\V{f}(\bx^{(t)})\ra \\
\leq ~& 2\eta^{(t)}_p \
\|\bp^{(t+1)} - \bp^{(t)}\|_1 \|\V{f}(\bx^{(t)})\|_{\infty}
\leq 2\eta^{(t)}_p R_f \|\bp^{(t+1)} - \bp^{(t)}\|_1, 
\end{aligned}
\] 
which gives us $\|\bp^{(t+1)} - \bp^{(t)}\|_1 \leq 2\eta^{(t)}_p R_f$,
where $(i)$ is because of \emph{Pinsker's inequality} 
(Theorem 4.19 in \cite{boucheron2013concentration}) and 
$(ii)$ is because of \autoref{lemma:noncvxproxsmoothgrad_dl}. 
To conclude, we have upper-bound $\calT^t_2$ as 
\[
\calT^t_2 \leq 2\eta^{(t)}_p R_f^2. 
\]
Then we finish the proof as 
\[
\begin{aligned}
&\sum_t \|\bx^{(t+1)} - \bx^{(t)}\|^2_2 \leq \
\dfrac{2\calL(\bp^{(0)}, \bx^{(0)}) - 2\calL(\bp^{(T+1)}, \bx^{(T+1)})}{L_h} + \
\dfrac{4 R_f^2 \sum_t \eta^{(t)}_p }{L_h} \\
\stackrel{(i)}{\leq} ~&\dfrac{2\calL(\bp^{(0)}, \bx^{(0)})}{L_h} + \
\dfrac{4 R_f^2 \sum_t \eta^{(t)}_p }{L_h}, 
\end{aligned} 
\]	
where $(i)$ is because $\calL(\bp^{(T+1)}, \bx^{(T+1)}) \geq 0$. 
\end{proof}

%=================================
\begin{lemma}
\label{lemma:noncvxproxsmoothgrad_dl}
With Alg.~\autoref{alg:non_cvx_prox_mw}, we have 
\[
D_{KL}\left(\bp^{(t+1)}||\bp^{(t)}\right) \leq \
\eta^{(t)}_p \la \bp^{(t)} - \bp^{(t+1)},~\V{f}(\bx^{(t)})\ra, 
\]
where $\V{f}(\bx^{(t)})$ denotes the vector whose 
$i$th element is $f_i(\bx^{(t)})$.
\end{lemma}

\begin{proof}
Here we have
\[
\begin{aligned}
&D_{KL}(\bp^{(t+1)}||\bp^{(t)})
= \sum_{i}p^{(t+1)}_i\
\log \left(\dfrac{p^{(t+1)}_i}{p^{(t)}_i}\right) \\
=~& \sum_{i}p^{(t+1)}_i\log \dfrac{e^{-\eta^{(t)}_p f_i (x_t)}}{Z_t} 
=-\log\left(Z_t \right) - \eta^{(t)}_p \sum_{i}p^{(t+1)}_i f_i(\bx^{(t)}) \\
=~& - \log\left(Z_t \right) - \
\eta^{(t)}_p \la \bp^{(t)}, ~\V{f}(\bx^{(t)})\ra - \
\eta^{(t)}_p \la \bp^{(t+1)} - \bp^{(t)},~\V{f}(\bx^{(t)})\ra,
\end{aligned}
\]
where $Z_t \triangleq \sum_{i} p^{(t)}_i e^{-\eta_p f_i (\bx^{(t)})}$. 
\par 
Then we have 
\[
\begin{aligned}
&\eta^{(t)}_p \la \bp^{(t)}- \bp^{(t+1)},~\V{f}(\bx^{(t)})\ra = \
D_{KL}(\bp^{(t+1)}||\bp^{(t)}) + \
\log\left(\sum_{i} p^{(t)}_i e^{-\eta^{(t)}_p f_i (\bx^{(t)})}\right) + \
\eta^{(t)}_p \la \bp^{(t)}, f(\bx_t)\ra  \\
\stackrel{(i)}{\geq} ~& D_{KL}(\bp^{(t+1)}||\bp^{(t)})+ \
\log\left[\prod_i e^{-\eta^{(t)}_p p^t_i f_i(\bx^{(t)})}\right] + \
\eta^{(t)}_p \la \bp^{(t)}, \V{f}(\bx^{(t)})\ra \\
=~& D_{KL}(\bp^{(t+1)}||\bp^{(t)}) + \
\underbrace{\sum_{i} \log\left(e^{-\eta^{(t)}_p p^{(t)}_i {f}_i(\bx^{(t)})}\right) + \
\eta^{(t)}_p \la \bp^{(t)}, f(\bx_t)\ra}_{0}  = \
D_{KL}(\bp^{(t+1)}||\bp^{(t)}),
\end{aligned}
\]
where in $(i)$ we use 
$\sum_i p_i x_i \geq \prod_{i}x_i^{p_i}$ such that 
$\sum_{i} p_i = 1, p_i \geq 0$.

\end{proof}

\section{Proof of \autoref{thm:noncvxproxregularsaddle}}

%\noncvxproxregularsaddlethm*

\begin{proof}
Define $\bp^{*}$ and $\bx^{*}$ as 
\begin{equation}
\bp^{*} = \argmin_{\bp}\sum_t \+{LR}\left(\bp, \bx^{(t)}\right),~~~
\bx^{*} = \argmin_{\bx}\sum_t \+{LR}\left(\bp^{(t)}, \bx\right).
\end{equation}
respectively
First we define $\calT^t_1$ and $\calT^t_2$ as
\[
\begin{aligned}
\calT^t_1 &= \+{LR}(\bp^{(t)}, \bx^{(t)}) - \+{LR}(\bp^{*}, \bx^{(t)}); \\
\calT^t_2 &= \+{LR}(\bp^{(t)}, \bx^{(t)}) - \+{LR}(\bp^{(t)}, \bx^{*}),	
\end{aligned}
\]
respectively. Then our goal becomes 
bounding $|\sum_t \calT^t_1| + |\sum_t \calT^t_2|$, 
For term $|\sum_t \calT^t_2|$, the analysis stays the same 
as \autoref{lemma:noncvx_prox_T2bound}. 
Here we focus on bounding $\sum_t \calT_1^t$, which proceeds as 
\[
\begin{aligned}
& \sum_t \calT_1^t = \
\sum_t \left(\+{LR}(\bp^{(t)}, \bx^{(t)}) - \+{LR}(\bp, \bx^{(t)})\right) = \
\sum_t \left\{-\la \underbrace{\nabla_t \+{LR}(\bp^{(t)}, \bx^{(t)})}_{\bg^{(t)}},~\bp - \bp^{(t)}\ra - \
\dfrac{\lambda_3}{2}\|\bp^{(t)} - \bp \|^2_2 \right\} \\
=~& \sum_t \left\{\la \bg^{(t)}, \bp^{(t)} - \bp\ra - \
\dfrac{\lambda_3}{2}\|\bp^{(t)} - \bp \|^2_2 \right\} 
\end{aligned} 
\]
Then we consider the distance $\|\bp^{(t+1)} - \bp\|^2_2$ which is 
\[
\begin{aligned}
&\|\bp^{(t+1)} - \bp\|^2_2 = \
\left\|\Proj_{\Delta}(\bp^{(t)} -\eta_p^t \bg^{(t)}) - \bp\right\|^2_2 \
\stackrel{(i)}{\leq} \
\left\|\bp^{(t)} - \eta_p^{(t)} \bg^{(t)} - \bp\right\|^2_2  \\
= ~& \left\|\bp^{(t)} - \bp\right\|^2_2 + (\eta_p^{(t)})^2\
\|\bg^{(t)}\|^2_2 - \
2\eta_p^{(t)} \la \bg^{(t)},~\bp^{(t)} - \bp\ra, 
\end{aligned}
\]
where in $(i)$ we use the contraction property for 
projection, 
which gives us 
\[
\la \bg^{(t)}, \bp^{(t)} - \bp\ra \leq \dfrac{\eta_p^{(t)} \|\bg^{(t)}\|^2_2}{2} + \
\dfrac{\|\bp^{(t)} - \bp\|^2_2 - \|\bp^{(t+1)} - \bp\|^2_2}{2\eta_p^{(t)}}
\]
By setting $\eta_p^{(t)} = (\lambda t)^{-1}$, we have 
\[
\begin{aligned}
&\sum_{t}\calT_1^t \leq \
\dfrac{R_g^2 \log T}{2\lambda_3} + \
\dfrac{\lambda_3}{2}\sum_{t}  t(\|\bp^{(t)} - \bp\|^2_2 - \|\bp^{(t+1)} - \bp\|^2_2)- \
\dfrac{\lambda_3}{2}\sum_t \|\bp^{(t)} - \bp\|^2_2 \\
=~& \dfrac{R_g^2 \log T}{2\lambda_3} + \
\dfrac{\lambda_3}{2}\sum_t \|\bp^{(t)} - \bp\|^2_2 - \
\dfrac{\lambda_3 (T+1)}{2}\|\bp^{(T+1)} - \bp\|^2_2 - \
\dfrac{\lambda_3}{2}\sum_t \|\bp^{(t)} - \bp\|^2_2 \\
= ~& \dfrac{R_g^2 \log T}{2\lambda_3} - \
\dfrac{\lambda_3 (T+1)}{2}\|\bp^{(T+1)} - \bp\|^2_2  \leq \
\dfrac{R_g^2 \log T}{2\lambda_3}.
\end{aligned}
\]
Hence, we have 
\[
\dfrac{\sum_t \calT_1^t + \calT_2^t}{T} \leq \
\dfrac{R_g^2 \log T}{2\lambda_3 T} + \dfrac{R^2}{2\eta_x T}, 
\]
which completes the proof. 
	
\end{proof}

\section{Proof of \autoref{thm:noncvxproxregularconverge}}
\label{thm_proof:noncvxproxregularconverge}

\begin{proof}
Following the same procedure in 
\autoref{thm_proof:noncvxproxsmoothgrad}, we
can bound 
\[
\begin{aligned}
&\|\bx^{(t+1)} - \bx^{(t)}\|^2_2 \leq \dfrac{2}{L_h}\
\left[\+{LR}(\bp^{(t)}, \bx^{(t)}) - \+{LR}(\bp^{(t)}, \bx^{(t+1)})\right] \\
=~& \
\dfrac{2}{L_h}\left[\underbrace{\+{LR}(\bp^{(t)}, \bx^{(t)}) - \+{LR}(\bp^{(t+1)}, \bx^{(t+1)})}_{\calT^t_1} \
+ \underbrace{\+{LR}(\bp^{(t+1)}, \bx^{(t+1)}) - \+{LR}(\bp^{(t)}, \bx^{(t+1)})}_{\calT^t_2}\right].
\end{aligned}
\]
Since $\calT_1^t$ will cancel themselves after summarization, 
we focus on bounding $\calT_2^t$. 
Then we have 
\[
\begin{aligned}
&\+{LR}(\bp^{(t)}, \bx^{(t+1)}) = 
\+{LR}(\bp^{(t+1)}, \bx^{(t+1)})  \
 + \la \nabla_{\bp} \+{LR}(\bp^{(t+1)}, \bx^{(t+1)}),\
 ~\bp^{(t)} - \bp^{(t+1)}\ra + \
\dfrac{\lambda_3}{2}\|\bp^{(t+1)} - \bp^{(t)}\|^2_2 \\
\stackrel{(i)}{=} ~&\
\+{LR}(\bp^{(t+1)}, \bx^{(t+1)}) + \la \V{f}(\bx^{(t+1)}) + \lambda_3(\bp^{(t+1)} - q),~\bp^{(t)} - \bp^{(t+1)}\ra + \
\dfrac{\lambda_3}{2}\|\bp^{(t+1)} - \bp^{(t)}\|^2_2, 
\end{aligned}
\]
where in $(i)$ we have 
\[
\nabla_{\bp}~\+{LR}(\bp^{(t+1)}, \bx^{(t+1)}) = \
\V{f}(\bx^{(t+1)}) + \lambda_3\left(\bp^{(t+1)} - \V{q}\right).
\]
Then we have 
\[
\begin{aligned}
&\+{LR}(\bp^{(t+1)},\bx^{(t+1)}) - \+{LR}(\bp^{(t)}, \bx^{(t+1)}) = \
\la \bg^{(t+1)},~\bp^{(t)} - \bp^{(t+1)}\ra +\
\dfrac{\lambda_3}{2}\|\bp^{(t+1)} - \bp^{(t)}\|^2_2 \\
\leq ~& \
\|\bg^{(t+1)}\|_2 \left\|\bp^{(t+1)} - \bp^{(t)}\right\|_2 + \
\dfrac{\lambda_3}{2} \left\|\bp^{(t+1)} - \bp^{(t)}\right\|^2_2 \\
\leq ~& \
\|\bg^{(t+1)}\|_2  \|\eta_p^{(t)}\bg^{(t)}\|_2 + \
\frac{\lambda_3 \left(\eta^{(t)}_p\right)^2}{2}\|\bg^{(t)}\|^2_2 \leq \
R_g^2\left(\eta^{(t)}_p + \frac{\lambda_3 \left(\eta^{(t)}_p\right)^2}{2}\right).
\end{aligned} 
\]
Hence, we conclude that 
\[
\sum_t \|\bx^{(t+1)} - \bx^{(t)}\|^2_2 \leq \
\frac{2\+{LR}(\bp^{(0)}, \bx^{(0)})}{L_h} + \dfrac{2 R_g^2}{L_h}\
\sum_{t}\left(\eta^{(t)}_p + \frac{\lambda_3 \left(\eta^{(t)}_p\right)^2}{2}\right), 
\] 
where $\|\bg^{(t)}\|_2 \leq R_g$. 

\end{proof}
	
\end{appendices}

\end{document}